\documentclass[letterpaper,11pt,reqno]{amsart}

\makeatletter
\usepackage{amssymb}
\usepackage{latexsym}
\usepackage{amsbsy}
\usepackage{amsfonts}
\usepackage{hyperref}
\usepackage{caption}
\usepackage{scalerel}
\usepackage{subcaption}
\usepackage{graphicx}
\usepackage{enumerate}
\usepackage{enumitem}
\usepackage{mathtools}
\usepackage{color}

\usepackage{tikz}

\def\marginpar#1{\ignorespaces}

\DeclareMathOperator\card{card}

\newtheorem{theorem}{Theorem}[section]
\newtheorem{lemma}[theorem]{Lemma}
\newtheorem{proposition}[theorem]{Proposition}
\newtheorem{corollary}[theorem]{Corollary}
\newtheorem{definition}[theorem]{Definition}

\numberwithin{equation}{section}
\makeatother
\begin{document}
%\title[Exchange PoS]{Exchange in the Proof of Stake protocol \\-- a continuous-time control approach}
\title[Weighted exchangeability]{Finite and infinite weighted exchangeable sequences}

\author[Wenpin Tang]{{Wenpin} Tang}
\address{Department of Industrial Engineer and Operations Research, Columbia University. %Email: 
} \email{wt2319@columbia.edu}

\date{\today} 
\begin{abstract}
Motivated by recent interests in predictive inference under distribution shift,
we study the problem of approximating finite weighted exchangeable sequences by 
a mixture of finite sequences with independent terms. 
Two different approaches are developed: 
a combinatorial approach and a probabilistic approach.
Various bounds are derived in terms of weight functions,
extending previous results on finite exchangeable sequences. 
As a byproduct, 
we obtain a version of de Finetti's theorem for
infinite weighted exchangeable sequences.
\end{abstract}

\maketitle



%\setcounter{tocdepth}{1}
%\tableofcontents

%------------------------------------------------------------------------------------------------
\section{Introduction}

\quad Exchangeability is a fundamental concept in 
probability theory and statistics, 
right next to the notion of independent and identically distributed (i.i.d.) random objects. 
It features a variety of research topics, including
exchangeable random graphs and networks \cite{BC09, DJ08, OR14},
random walk with exchangeable increments \cite{SA53, PT23}
exchangeable random partitions and clustering \cite{CCB18, Pit95},
and Bayesian inference \cite{HW23, KD18, Rubin78}, just to name a few. 
Perhaps the best known result is de Finetti's representation of 
an infinite exchangeable binary sequence \cite{deF29}: 
an exchangeable sequence $X_1, X_2, \ldots \in \{0,1\}$ 
can be constructed by:
\begin{equation*}
\mbox{sample } p \sim \mu, \mbox{ then draw } X_1, X_2, \ldots \stackrel{i.i.d.}{\sim} \mbox{ Bernoulli}(p),
\end{equation*}
for some distribution $\mu$ on $[0,1]$. 
In words, any infinite exchangeable binary sequence is a mixture of i.i.d. Bernoulli sequences. 
The result was extended by Hewitt and Savage \cite{HS55} to a general space $\mathcal{X}$. 
See \cite{Aldous85, Kallen05} for a comprehensive review on the topic of exchangeability.

\quad Recently, \cite{BC23, TR19} proposed the notion of {\em weighted exchangeability} 
to deal with covariate shift in predictive inference. 
Roughly speaking, the sequence $X_1, \ldots, X_n$ on $\mathcal{X}$ is weighted exchangeable with respect to 
the weight functions $\lambda_1, \ldots, \lambda_n$ if 
the probability mass/density function $f_n$ (if $\mathcal{X}$ is discrete/continuous) 
can be factorized as:
\begin{equation}
\label{eq:factor}
f_n(x_1, \ldots, x_n) = \left(\prod_{i=1}^n \lambda_i(x_i)\right) g_n(x_1, \ldots, x_n),
\end{equation}
where $g_n: \mathcal{X}^n \to (0,\infty)$ is symmetric, 
i.e., $g_n(x_1, \ldots, x_n) = g_n(x_{\sigma(1)}, \ldots, x_{\sigma(n)})$.
For instance, if $\lambda_i \equiv 1$ then we are in the exchangeable setting. 
The form \eqref{eq:factor} also appeared
in the context of extremal families of a statistical model,
where $\mathcal{X} = \{0,1\}$ \cite{Lau88}.
See Definition \ref{def:weightex} for details, 
and an (obvious) extension to infinite sequences.  

\quad Among many possible questions on weighted exchangeability,
a natural first one is whether there is a de Finetti-like theorem for weighted exchangeable sequences.
That is, can a weighted exchangeable sequence be represented or approximated 
by a mixture of weighted i.i.d. (i.e., independent but not necessarily identically distributed) sequences?
Intuitively, the weight functions $\lambda_1, \lambda_2, \ldots$ should play a role 
in the answer to this question.
\begin{itemize}[itemsep = 3 pt]
\item
In \cite{Lau88}, it was proved that for $\mathcal{X} = \{0,1\}$,
de Finetti's theorem for infinite weighted exchangeable sequences holds {\em if and only if}
\begin{equation}
\sum_{i = 1}^\infty \frac{\min (\lambda_i(0), \lambda_i(1))}{\max (\lambda_i(0), \lambda_i(1))} = \infty.
\end{equation}
\item
In \cite{BC23}, it was shown that for a fairly general $\mathcal{X}$, 
de Finetti's theorem for infinite weighted exchangeable sequences holds {\em if}
\begin{equation}
\label{eq:gencond0}
\sum_{i = 1}^\infty r_i = \infty, \quad
\mbox{where } 
r_i:=\frac{\inf_{\mathcal{X}} \lambda_i(x)}{\sup_{\mathcal{X}} \lambda_i(x)}.
\end{equation}
It is unknown whether the condition \eqref{eq:gencond0} is necessary except when $\mathcal{X} = \{0,1\}$.
\end{itemize}

\quad The proofs in \cite{BC23} rely heavily on  measure theoretical machinery,
and that paper also explores the connection between different concepts
including the weighted de Finetti property, the weighted zero-one law (see also \cite{AP79}) and the weighted law of large numbers.

\quad This paper aims to provide a ``quantitative" approach to de Finetti's representation of
weighted exchangeable sequences,
and the perspective is different from \cite{BC23, Lau88}.
To this end, we consider {\em finite} weighted exchangeable sequences 
in the spirit of Diaconis and Freedman \cite{DF80}.
It is known that 
a finite (weighted) exchangeable sequence is not a mixture of finite (weighted) i.i.d. sequences,
and the goal is to understand
how close the former can be approximated  by the latter. 
We will establish various estimates of the difference between sampling with and without replacement,
and show how they can be used to derive bounds in terms of the weight functions. 
Two different approaches will be developed: 
one is purely combinatorial involving the oscillation of the likelihood ratio between sampling with and without replacement,
and the other is probabilistic based on a thinning argument and stochastic comparisons.
As a corollary, 
we obtain a version of de Finetti's theorem for infinite weighted exchangeable sequences
by taking limits. 
We point out that 
the problem of quantifying the approximation of 
finite weighted exchangeable sequences by a mixture of weighted i.i.d. sequences 
is of independent interest.

\quad The rest of the paper is organized as follows. 
In Section \ref{sc2}, 
we provide background on weighted exchangeability.
We develop the combinatorial approach to the weighted de Finetti results in Section \ref{sc3}, 
and the probabilistic approach is presented in Section \ref{sc4}.

%------------------------------------------------------------------------------------------------
\section{Background on weighted exchangeability}
\label{sc2}

\quad Throughout this paper, we consider random variables on a Polish space $\mathcal{X}$, 
equipped with the Borel $\sigma$-algebra $\mathcal{B} = \mathcal{B}(\mathcal{X})$.  
Also define $\mathcal{X}^n$ and $\mathcal{X}^{\infty}$ to be 
the $n$-fold and countably infinite product spaces,
with respective Borel $\sigma$-algebras 
$\mathcal{B}^n = \mathcal{B}(\mathcal{X})^n$ and $\mathcal{B}^{\infty} = \mathcal{B}(\mathcal{X})^{\infty}$.
Denote by $\Lambda$ the set of measurable functions $\mathcal{X}$ to $(0,\infty)$,
and $\Lambda^n$ and $\Lambda^{\infty}$ the corresponding $n$-fold and countably infinite product spaces.
Let $\mathfrak{S}_n$ be the set of permutations on $[n]:= \{1,\ldots, n\}$,
and for $\sigma \in \mathfrak{S}_n$, we denote
$\sigma[i,j]:= (\sigma(i), \ldots, \sigma(j))$ for $1 \le i \le j \le n$.

\quad Below we define the notion of weighted exchangeability for a sequence of random variables. 

\begin{definition} 
\label{def:weightex}
\cite{BC23, TR19}
\begin{enumerate}[itemsep = 3 pt]
\item
(Finite weighted exchangeability).
Let $n$ be a positive integer, 
and $P_n$ be the distribution of $(X_1, \ldots X_n) \in \mathcal{X}^n$.
We say that $P_n$ is weighted exchangeable if 
there is $\lambda = (\lambda_1, \ldots, \lambda_n) \in \Lambda^n$ such that 
the measure $\overline{P}_n$ defined by 
\begin{equation}
\label{eq:changem}
\overline{P}_n(B) = \int_B \frac{dP_n(x_1, \ldots, x_n)}{\prod_{i = 1}^n \lambda_i(x_i)}, \quad \mbox{for } B \in \mathcal{B}^n,
\end{equation}
 is an exchangeable probability measure. 
The sequence $(X_1, \ldots X_n)$,
or the distribution $P_n$ is called $\lambda$-exchangeable.
\item
(Infinite weighted exchangeability).
Let $P$ be the distribution of $(X_1, X_2, \ldots) \in \mathcal{X}^{\infty}$,
and $P_n$ be the marginal measure of $P$ on $(\mathcal{X}^n, \mathcal{B}^n)$ defined by
$P_n(B): = P(B \times \mathcal{X} \times \mathcal{X} \times \ldots)$ for $B \in \mathcal{B}^n$.
We say that $P$ is weighted exchangeable if 
there is $\lambda = (\lambda_1, \lambda_2, \ldots) \in \Lambda^{\infty}$ such that
for each $n \ge 1$, the distribution $P_n$ is $\lambda_n$-exchangeable,
where  $\lambda_n: = (\lambda_1, \ldots, \lambda_n) \in \Lambda^n$.
The sequence $(X_1, X_2, \ldots)$, or the distribution $P$ is called $\lambda$-exchangeable.
\end{enumerate}
\end{definition}

\quad Let's make a few remarks.
First, Definition \ref{def:weightex} is a slight adaptation of \cite[Definition 3 \& 4]{BC23},
where the notion of weighted exchangeability is defined for general measures 
in the following obvious way.
A measure $Q_n$ on $(\mathcal{X}^n, \mathcal{B}^{n})$ is called $\lambda$-exchangeable if the measure $\overline{Q}_n$ defined by 
\begin{equation*}
\overline{Q}_n(B) = \int_B \frac{dQ_n(x_1, \ldots, x_n)}{\prod_{i = 1}^n \lambda_i(x_i)}, \quad \mbox{for } B \in \mathcal{B}^n,
\end{equation*}
is exchangeable in the sense that 
$\overline{Q}_n(B_1 \times \ldots \times B_n) = \overline{Q}_n(B_{\sigma(1)} \times \ldots \times B_{\sigma(n)})$
for any $B_1, \ldots, B_n \in \mathcal{B}$ and any permutation $\sigma \in \mathfrak{S}_n$.
Similarly, a measure $Q$ on $(\mathcal{X}^{\infty}, \mathcal{B}^{\infty})$ is $\lambda$-exchangeable if
for each $n$, its marginal measure $Q_n$ is $\lambda_n$-exchangeable. 
Here we focus on weighted exchangeable {\em probability measures},
which are useful for practical purposes.
Most results can be easily extended to finite measures.

\quad Second, the key ingredient \eqref{eq:changem} in the definition of weighted exchangeability 
is basically the factorization \eqref{eq:factor}.
In other words, 
weighted exchangeable sequences can be viewed as 
a change of measure of exchangeable sequences,
with {\em separable} Radon-Nikodym derivatives.

\quad Finally, the weights $\lambda$ are not unique for a weighted exchangeable sequence.
For instance, consider $(X_1, X_2)$ with density
$f(x_1, x_2) = x_1 x_2$ on $[0,1]^2$.
Note that $x_1 x_2 = x_1 \cdot x_2 \cdot 1 = 1 \cdot 1 \cdot x_1 x_2$,
so $(X_1, X_2)$ is $\lambda$-exchangeable
with $\lambda_1(x) = \lambda_2(x) \equiv 1$ or $\lambda_1(x) = \lambda_2 (x) = x$.
For $\lambda \in \Lambda^n$ and $\theta = (\theta_1, \ldots, \theta_n) \in (0, \infty)^n$,
define $\lambda_\theta: = \theta \cdot \lambda = (\theta_1 \lambda_1, \ldots, \theta_n \lambda_n)$.
Obviously, if $P_n$ is $\lambda$-exchangeable, 
then $P_n$ is  $\lambda_\theta$-exchangeable provided $\prod_{i = 1}^n \theta_i = 1$.
Thus, the notion of $\lambda$-exchangeability is defined up to equivalent classes on $\lambda$. 

\quad de Finetti's theorem states that any infinite exchangeable sequence 
is a mixture of i.i.d. sequences (sampling with replacement),
see \cite[Section 3]{Aldous85}, \cite[Theorem 1.1]{Kallen05}.
A finite exchangeable sequence 
needs not be a mixture of i.i.d. sequences; 
instead it is a mixture of urn sequences (sampling without replacement),
see \cite{DF80}, \cite[Proposition 1.8]{Kallen05}.
Recall that an element in a convex set is called {\em extreme} if 
it cannot be expressed as a strict mixture of the elements in the set.
The following proposition records some useful results on 
finite weighted exchangeable sequences. 

\begin{proposition}
\label{prop:finiteexrep}
Let $\lambda  \in \Lambda^n$, and $(X_1, \ldots, X_n) \in \mathcal{X}^n$ be $\lambda$-exchangeable.
\begin{enumerate}[itemsep = 3 pt]
\item
The conditional distribution of $(X_1, \ldots, X_n)$ given the empirical measure $\widehat{P}_n:= \frac{1}{n}\sum_{i = 1}^n \delta_{X_i}$
(or its possibly repeated values $\mathcal{U}_X:=\{X_1, \ldots, X_n\}$) is 
\begin{equation}
\label{eq:conddist}
(X_1, \ldots, X_n  \,|\,  \widehat{P}_n) 
\stackrel{d}{=} \sum_{\sigma} \frac{\prod_{i = 1}^n \lambda_i(X_{\sigma(i)})}{\sum_\sigma \prod_{i = 1}^n \lambda_i(X_{\sigma(i)})} \delta_{X_{\sigma(1)}, \ldots, X_{\sigma(n)}},
\end{equation}
where $\sigma \in \mathfrak{S}_n$ and the sum is over all permutations in $\mathfrak{S}_n$.
Denote by $P_{\mathcal{U}_X, n}$ the conditional distribution given by \eqref{eq:conddist}.
\item
The set of $\lambda$-exchangeable sequences is convex,
and each extreme $\lambda$-exchangeable distribution is of form \eqref{eq:conddist}
for some non-random empirical measure $\frac{1}{n}\sum_{i = 1}^n \delta_{x_i}$.
Consequently, the distribution of $(X_1, \ldots, X_n)$ is a unique mixture of extreme points:
\begin{equation}
P_n(\cdot) = \int w_{\scaleto{\mathcal{U}_x}{5 pt}} P_{\mathcal{U}_x, n}(\cdot)  \, d \, \mathcal{U}_x, \quad 
\mbox{with } w_{\scaleto{\mathcal{U}_x}{5 pt}} \ge 0 \mbox{ and } \int w_{\scaleto{\mathcal{U}_x}{5 pt}} d \, \mathcal{U}_x= 1.
\end{equation}
\end{enumerate}
\end{proposition}
\begin{proof}
(1) It was proved in \cite[Proposition 6]{BC23}  
the conditional (marginal) distribution of $X_i$ given $\frac{1}{n}\sum_{i = 1}^n \delta_{X_i}$ is
\begin{equation}
\label{eq:margdist}
(X_i \,|\, \widehat{P}_n ) \stackrel{d}{=} 
\sum_{j= 1}^n \frac{\sum_{\sigma(i) = j} \prod_{i = 1}^n \lambda_i(X_{\sigma(i)})}{\sum_\sigma \prod_{i = 1}^n \lambda_i(X_{\sigma(i)})} \delta_{X_j},
\end{equation}
and the formula \eqref{eq:conddist} is an easy extension that can be read from the proof. 
Also note that the joint distribution \eqref{eq:conddist} is consistent with \eqref{eq:margdist}
by summing over all permutations $\sigma \in \mathfrak{S}_n$ with $\sigma(i) = j$.

(2) The fact that the set of $\lambda$-exchangeable distributions is convex is straightforward from Definition \ref{def:weightex},
see \cite[Proposition 2]{BC23}. 
The characterization of extreme $\lambda$-exchangeable distributions follows from 
the general machinery as in \cite[Proposition 1.4 \& 1.8]{Kallen05}.
\end{proof}

\quad Several remarks are in order.
First, it is clear from the formula \eqref{eq:conddist} that the conditional distribution of $(X_1, \ldots, X_k)$ for $k \le n$
given $\frac{1}{n}\sum_{i = 1}^n \delta_{X_i}$ is
\begin{equation}
\label{eq:conddist2}
\begin{aligned}
& (X_1, \ldots, X_k  \,|\, \widehat{P}_n )  \\
& \qquad \qquad \stackrel{d}{=}  \sum_{j_1, \ldots, j_k}\frac{\sum_{\sigma(1) = j_1, \ldots, \sigma(k) = j_k}\prod_{i = 1}^n \lambda_i(X_{\sigma(i)})}{\sum_\sigma \prod_{i = 1}^n \lambda_i(X_{\sigma(i)})}
\delta_{X_{j_1}, \ldots, X_{j_k}}, \\
& \qquad \qquad \stackrel{d}{=} \sum_{j_1, \ldots, j_k} \left(\prod_{i = 1}^k \lambda_i(X_{j_i})\right) \frac{\sum_{\sigma[k+1,n] = [n] \setminus \{j_1, \ldots, j_k\}}\prod_{i = k+1}^n \lambda_i(X_{\sigma(i)})}{\sum_\sigma \prod_{i = 1}^n \lambda_i(X_{\sigma(i)})}
\delta_{X_{j_1}, \ldots, X_{j_k}},
\end{aligned}
\end{equation}
where $j_1,\ldots, j_k$ take $k$ distinct values from $[n]$.
Denote by $\mathcal{P}_{\mathcal{U}_X, k}$ the conditional distribution in \eqref{eq:conddist2}.
There are $(n)_k: = n!/(n-k!)$ terms in the sum $\sum_{j_1, \ldots, j_k}$,
and $(n-k)!$ terms in $\sum_{\sigma[k+1,n] = [n] \setminus \{j_1, \ldots, j_k\}}$.
The second part of Proposition \ref{prop:finiteexrep} implies that
each extreme point of the $k$-marginal $\lambda$-exchangeable distributions 
is of form $\mathcal{P}_{\mathcal{U}_x, k}$ for some non-random $\frac{1}{n} \sum_{i = 1}^n \delta_{x_i}$.
This way, the formula \eqref{eq:conddist2} will be very useful in deriving an error bound
between the distribution of $(X_1, \ldots, X_k)$ and its closest mixture of i.i.d. sequences. 

\quad Second, if $\lambda_i \equiv 1$ for all $i$ (exchangeable), then the formulas \eqref{eq:conddist} and \eqref{eq:conddist2} specialize to
\begin{equation}
\label{eq:conddist3}
(X_1, \ldots, X_k  \,|\,  \widehat{P}_n ) \stackrel{d}{=} \frac{1}{(n)_k} \sum_{j_1, \ldots, j_k} \delta_{X_{j_1}, \ldots, X_{j_k}}, \quad k \le n,
\end{equation}
which is exactly sampling without replacement from the ``urn" $\mathcal{U}_X$.
So $(X_1, \ldots, X_n)$ given the urn $\mathcal{U}_X$ can be sampled sequentially: 
pick $X_{j_1}$ uniformly at random from $\mathcal{U}_X$,
and pick $X_{j_2}$ uniformly at random from $\mathcal{U}_X \setminus \{X_{j_1}\}$, and so on. 
But for general weights $\lambda$, 
there does not seem to be an easy way to sample $(X_1, \ldots, X_n)$ from $\mathcal{U}_X$.

\quad Finally, the extreme points of $k$-marginal exchangeable distributions can be made explicit. 
By \cite[Theorem 4.5]{CFV23}, 
each extreme point is of form
\begin{equation}
\label{eq:expansion}
\frac{n^k}{(n)_k} \widehat{P}_n^{\otimes} + \varepsilon_{n,k}(\widehat{P}_n)
= \widehat{P}_n^{\otimes} + \left( \left(\frac{n^k}{(n)_k} - 1 \right) \widehat{P}_n^{\otimes}  + \varepsilon_{n,k}(\widehat{P}_n)\right),
\end{equation}
where $\varepsilon_{n,k}(\widehat{P}_n)$ is a ``polynomial" involving alternating sums of 
pushforward and symmetrization of $\widehat{P}_n$ (see \cite[(4.10)]{CFV23}, and also \cite{Bob05}).
The term $\left(\frac{n^k}{(n)_k} - 1 \right) \widehat{P}_n^{\otimes}  + \varepsilon_{n,k}(\widehat{P}_n)$
gives the difference between the distribution \eqref{eq:conddist3}
and $k$ i.i.d. samples from the empirical measure $\widehat{P}_n$.
It is also possible to derive a formula analogous to \eqref{eq:expansion}
for weighted exchangeable distributions,
but we do not pursue it here and leave it to the readers. 

%------------------------------------------------------------------------------------------------
\section{A combinatorial approach}
\label{sc3}

\quad In this section, we present a combinatorial approach to approximating finite weighted exchangeable sequences
by a mixture of weighted i.i.d. (or independent but not identically distributed) sequences. 
Recall that the total variation distance between two distributions $P$ and $Q$ is 
$|P - Q|_{TV}: = \sup_{A} |P(A) - Q(A)|$.

\begin{theorem}
\label{thm:general}
Let $P_n$ be the distribution of a $\lambda$-exchangeable sequence $(X_1, \ldots, X_n)$ on $\mathcal{X}$,
and $P_k$ be the marginal distribution of $(X_1, \ldots, X_k)$ for $k \le n$.
Assume that
\begin{equation}
\label{eq:assumpr}
r_i: = \frac{\inf_{\mathcal{X}} \lambda_i(x)}{\sup_{\mathcal{X}} \lambda_i(x)} > 0, \quad \mbox{for } i \in [n].
\end{equation}
Then there exists a mixture of weighted i.i.d. distributions $Q_n$ on $\mathcal{X}^n$, 
with $Q_{n,k}$ the marginal distribution on its first $k$ coordinates such that
\begin{equation}
\label{eq:boundgen}
|P_k - Q_{n,k}|_{TV} \le 
\frac{k(k-1)}{2 n}\left(\prod_{i = 1}^k r_i \right)^{-1} + \frac{(\prod_{i = k+1}^n r_i)^{-\frac{1}{2}}- 1}{(\prod_{i = k+1}^n r_i)^{-\frac{1}{2}} +1}.
\end{equation}
\end{theorem}
\begin{proof}
By Proposition \ref{prop:finiteexrep} (2), $P_n$ is of form
$P_n = \int w_{\scaleto{\mathcal{U}_x}{5 pt}} P_{\mathcal{U}_x, n} \, d \, \mathcal{U}_x$,
and $P_k = \int w_{\scaleto{\mathcal{U}_x}{5 pt}} P_{\mathcal{U}_x, k} \, d \, \mathcal{U}_x$.
Define the probability measure $Q_{\mathcal{U}_x, n}$ on $\mathcal{U}_x$ by
sampling $X_i$ according to $\frac{\lambda_i(\cdot)}{\sum_{j = 1}^n \lambda_i(x_j)}$,
independently for each $i$.
So 
\begin{equation}
\label{eq:condQ2}
Q_{\mathcal{U}_x, n} = \sum_{s(1), \ldots, s(n)} \left(\prod_{i=1}^n \frac{\lambda_i(x_{s(i)})}{\sum_{j = 1}^n 
\lambda_i(x_j)} \right) \delta_{x_{s(1)}, \ldots, x_{s(n)}},
\end{equation}
where the sum is over all $s(1), \ldots s(n) \in [n]$.
Clearly, $Q_{\mathcal{U}_x, n}$ is weighted i.i.d.
Let 
\begin{equation}
\label{eq:Qdef}
Q_n: = \int  w_{\scaleto{\mathcal{U}_x}{5 pt}} Q_{\mathcal{U}_x, n} \, d \, \mathcal{U}_x.
\end{equation}
We have 
\begin{equation}
\label{eq:urnreduction2}
\begin{aligned}
|P_k - Q_{n,k} |_{TV} & = \left|\sum_{\scaleto{\mathcal{U}_x}{5 pt}} w_{\scaleto{\mathcal{U}_x}{5 pt}} (P_{\mathcal{U}_x, k} - 
Q_{\mathcal{U}_x, k})\right|_{TV}  \\
& \le 
\sum_{\scaleto{\mathcal{U}_x}{5 pt}} w_{\scaleto{\mathcal{U}_x}{5 pt}} |P_{\mathcal{U}_x, k} - 
Q_{\mathcal{U}_x, k}|_{TV} \le \max_{\scaleto{\mathcal{U}_x}{5 pt}} |P_{\mathcal{U}_x, k} - 
Q_{\mathcal{U}_x, k}|_{TV}.
\end{aligned}
\end{equation}
The goal is to bound $|P_{\mathcal{U}_x, k} - Q_{\mathcal{U}_x, k}|_{TV}$ for a given urn $\mathcal{U}_x$.
Note that 
the support of $Q_{\mathcal{U}_x, k}$ is $\mathcal{U}_x^k$;
and the support of $P_{\mathcal{U}_x, k}$ is by picking $k$ elements without replacement from $\mathcal{U}_x$,
which we denote by $\widetilde{\mathcal{U}}_X^k$.
Importantly, $P_{\mathcal{U}_x, k}$ is generally not $Q_{\mathcal{U}_x, k}$ conditioned on $\widetilde{\mathcal{U}}_X^k$.

\quad Let $q:= Q_{\mathcal{U}_x, k}(\widetilde{\mathcal{U}}_X^k)$,
and $Q_{\mathcal{U}_x, k}^S:=Q_{\mathcal{U}_x,k}(\cdot \,|\, S)$ for any $S \subset \mathcal{U}_x^k$.
Since  $\widetilde{\mathcal{U}}_X^k \subset \mathcal{U}_x^k$, we write:
\begin{equation*}
Q_{\mathcal{U}_x, k} = qQ_{\mathcal{U}_x, k}^{\widetilde{\mathcal{U}}_X^k} + (1-q) Q_{\mathcal{U}_x, k}^{(\widetilde{\mathcal{U}}_X^k)^c}.
\end{equation*}
By convexity of total variation, we get:
\begin{equation}
\label{eq:TVconv}
|P_{\mathcal{U}_x, k} - Q_{\mathcal{U}_x, k}|_{TV} 
\le (1-q) + q \left|P_{\mathcal{U}_x, k} - Q_{\mathcal{U}_x, k}^{\widetilde{\mathcal{U}}_X^k}\right|_{TV}
\le (1-q) + \left|P_{\mathcal{U}_x, k} - Q_{\mathcal{U}_x, k}^{\widetilde{\mathcal{U}}_X^k}\right|_{TV}.
\end{equation}

{\em Bounding $(1-q)$}: Note that 
\begin{align}
\label{eq:TVformula}
1- q &= 1 -\sum_{\sigma(1), \ldots, \sigma(k)} \prod_{i=1}^k \frac{\lambda_i(x_{\sigma(i)})}{\sum_{j = 1}^n \lambda_i(x_j)}  \notag \\
& = \frac{\prod_{i = 1}^k \left( \sum_{j = 1}^n \lambda_i(x_j) \right) - \sum_{\sigma(1), \ldots, \sigma(k)} \prod_{i = 1}^k \lambda_i(x_{\sigma(i)})}{\prod_{i = 1}^k \left( \sum_{j = 1}^n \lambda_i(x_j) \right)},
\end{align}
where the sum $ \sum_{\sigma(1), \ldots, \sigma(k)}$ is over all $k$-tuples picking without replacement from $\mathcal{U}_X$
(or equivalentlly the restriction of $\sigma \in \mathfrak{S}_n$ to $[k]$).
There are $n^k - (n)_k$ terms in the numerator of \eqref{eq:TVformula},
each of which is bounded from above by $\prod_{i = 1}^k \sup_{\mathcal{X}} \lambda_i(x)$;
and there are $n^k$ terms in the denominator of \eqref{eq:TVformula}, 
each of which is bounded from below by $\prod_{i = 1}^k \inf_{\mathcal{X}} \lambda_i(x)$.
So
\begin{equation}
1-q  \le
\left(\prod_{i = 1}^k r_i \right)^{-1} \left( 1 - \frac{n_{(k)}}{n^k} \right),
\end{equation}
where $1 - \frac{n_{(k)}}{n^k}$ is the total variation distance
between sampling $k$ elements uniformly from $[n]$ with and without replacement.
By \cite{Freed77}, we have $1 - \frac{n_{(k)}}{n^k} \le \frac{k(k-1)}{2n}$,
which implies
\begin{equation}
\label{eq:TVest}
1 - q \le \frac{k(k-1)}{2n} \left(\prod_{i = 1}^k r_i \right)^{-1}.
\end{equation}

{\em Bounding $\left|P_{\mathcal{U}_x, k} - Q_{\mathcal{U}_x, k}^{\widetilde{\mathcal{U}}_X^k}\right|_{TV}$}:
By \eqref{eq:conddist2} and \eqref{eq:condQ2}, 
\begin{equation*}
\frac{d P_{\mathcal{U}_x, k}}{d Q_{\mathcal{U}_x, k}^{\widetilde{\mathcal{U}}_X^k}} \propto 
\sum_{\sigma[k+1,n] = [n] \setminus \{j_1, \ldots, j_k\}}\prod_{i = k+1}^n \lambda_i(X_{\sigma(i)})
: = H(j_1, \ldots, j_k).
\end{equation*}
A standard likelihood ratio estimate gives
\begin{equation}
\label{eq:TVk}
\left|P_{\mathcal{U}_x, k} - Q_{\mathcal{U}_x, k}^{\widetilde{\mathcal{U}}_X^k}\right|_{TV} \le \frac{\sqrt{R_H} - 1}{\sqrt{R_H} + 1},
\end{equation}
where 
$R_H: = \max_{\widetilde{\mathcal{U}}_X^k} H / \min_{\widetilde{\mathcal{U}}_X^k} H$
is the oscillation of the likelihood ratio.
Furthermore, each term in $H(j_1, \ldots, j_k)$ is between $\Pi_{i = k+1}^n \inf_{\mathcal{X}} \lambda_i(x)$
and $\Pi_{i = k+1}^n \sup_{\mathcal{X}} \lambda_i(x)$.
Thus, 
\begin{equation}
\label{eq:RH}
R_H \le \prod_{i = k+1}^n r_i^{-1}.
\end{equation}
Combining \eqref{eq:TVk} and \eqref{eq:RH} yields:
\begin{equation}
\label{eq:RH2}
\left|P_{\mathcal{U}_x, k} - Q_{\mathcal{U}_x, k}^{\widetilde{\mathcal{U}}_X^k}\right|_{TV} \le 
\frac{(\prod_{i = k+1}^n r_i)^{-\frac{1}{2}}- 1}{(\prod_{i = k+1}^n r_i)^{-\frac{1}{2}} +1}.
\end{equation}
Now \eqref{eq:TVconv}, \eqref{eq:TVest} and \eqref{eq:RH2}
yields the bound \eqref{eq:boundgen}.
\end{proof}

\quad If $P_n$ is exchangeable (i.e., $r_i =1$ for all $i$), 
the bound \eqref{eq:boundgen} reduces to the Diaconis-Freedman bound (\cite[Theorem 14]{DF80}, \cite[Proposition 1.9]{Kallen05}),
which was shown to be sharp (see \cite[Section 4]{DF80}). 
The condition \eqref{eq:assumpr} is required
so that the bound \eqref{eq:boundgen} is finite. 
The bound \eqref{eq:boundgen} indicates that
the more spread the weights $\lambda_i$'s are (or the smaller the ratios $r_i$'s are), 
the larger the error of approximation $|P_k - Q_k|_{TV}$ may be.
While we don't know if the factor $(\prod_{i = 1}^k r_i)^{-1}$ in the bound is sharp,
the order of error $k^2/n$ is optimal.
As we will see later, 
the factor $(\prod_{i=1}^k r_i)^{-1}$ 
turns out to be crucial in
deriving a de Finetti representation for infinite weighted exchangeable sequences
by passing to the limit.  

\quad Next we consider the special case where $\mathcal{X}$ is finite. 
The following result shows that a better bound $k/n$ can be obtained.

\begin{theorem}
\label{thm:finite}
Assume that $\card(\mathcal{X}) = c < \infty$.
Let $P_n$ be the distribution of a $\lambda$-exchangeable sequence $(X_1, \ldots, X_n)$ on $\mathcal{X}$,
and $P_k$ be the marginal distribution of $(X_1, \ldots, X_k)$ for $k \le n$.
Assume that
\begin{equation}
r_i: = \frac{\min_{\mathcal{X}} \lambda_i(x)}{\max_{\mathcal{X}} \lambda_i(x)} > 0, \quad \mbox{for } i \in [n].
\end{equation}
Then there exists a mixture of weighted i.i.d. distributions $Q_n$ on $\mathcal{X}^n$, 
with $Q_{n,k}$ the marginal distribution on its first $k$ coordinates such that
\begin{equation}
\label{eq:boundfinite}
|P_k - Q_{n,k}|_{TV} \le 
\frac{ck}{n} \frac{3 - \prod_{i=1}^k r_i}{2 \prod_{i=1}^k r_i} + \frac{(\prod_{i = k+1}^n r_i)^{-\frac{1}{2}}- 1}{(\prod_{i = k+1}^n r_i)^{-\frac{1}{2}} +1}.
\end{equation}
\end{theorem}
\begin{proof}
We follow the notations as in the proof of Theorem \ref{thm:general}.
By \eqref{eq:urnreduction2}, 
it suffices to bound $|P_{\mathcal{U}_x, k} - Q_{\mathcal{U}_x, k}|_{TV}$.
By triangle inequality and  \eqref{eq:RH2}, we get:
\begin{equation}
\label{eq:keytr}
\begin{aligned}
|P_{\mathcal{U}_x, k} - Q_{\mathcal{U}_x, k}|_{TV} 
& \le \left|P_{\mathcal{U}_x, k} - Q_{\mathcal{U}_x, k}^{\widetilde{\mathcal{U}}_X^k}\right|_{TV} + \left|Q_{\mathcal{U}_x, k}^{\widetilde{\mathcal{U}}_X^k} - Q_{\mathcal{U}_x, k}\right|_{TV} \\
& \le \frac{(\prod_{i = k+1}^n r_i)^{-\frac{1}{2}}- 1}{(\prod_{i = k+1}^n r_i)^{-\frac{1}{2}} +1} + \left|Q_{\mathcal{U}_x, k}^{\widetilde{\mathcal{U}}_X^k} - Q_{\mathcal{U}_x, k}\right|_{TV}.
\end{aligned}
\end{equation}
So it remains to bound $\left|Q_{\mathcal{U}_x, k}^{\widetilde{\mathcal{U}}_X^k} - Q_{\mathcal{U}_x, k}\right|_{TV}$.
Observe that
\begin{equation*}
Q_{\mathcal{U}_x, k}^{\widetilde{\mathcal{U}}_X^k}(dx) \propto \prod_{i=1}^k \lambda_i(x_i)P_{o,k}(dx), \quad
Q_{\mathcal{U}_x, k}(dx) \propto \prod_{i=1}^k \lambda_i(x_i)Q_{o,k}(dx),
\end{equation*}
where $P_{o,k}$ (resp. $Q_{o,k}$) denotes the $k$-marginal distribution of 
uniform sampling without replacement (resp. with replacement).
We need the following lemma to bound $\left|Q_{\mathcal{U}_x, k}^{\widetilde{\mathcal{U}}_X^k} - Q_{\mathcal{U}_x, k}\right|_{TV}$.
\begin{lemma}\label{lem:stab}
Let $\mu, \nu$ be two probability measure on $\mathcal{X}$,
and $w: \mathcal{X} \to \mathbb{R}$ be such that $0 <  \alpha \le w(x) \le 1$ for all $x$.
Define
\begin{equation*}
\mu^w(dx) \propto w(x) \mu(dx), \quad \nu^w(dx) \propto w(x) \nu(dx).
\end{equation*}
Then 
\begin{equation*}
|\mu^w - \nu^w|_{TV} \le \frac{3 - \alpha}{2 \alpha} |\mu - \nu|_{TV}.
\end{equation*}
\end{lemma}
\begin{proof}
Denote by $a: = \int w(x) \mu(dx)$ and $b:=\int w(x) \nu(dx)$.
We have:
\begin{align*}
2 |\mu^w - \nu^w|_{TV} & \le \frac{1}{a} \int w(x) |\mu(dx) - \nu(dx)| + \left| \frac{1}{a} - \frac{1}{b} \right| \int w(x) \nu(dx) \\
& = \frac{1}{a} \int w(x) |\mu(dx) - \nu(dx)| + \frac{|a-b|}{a} \\
& \le \frac{2 |\mu - \nu|_{TV}}{a} + \frac{|a-b|}{a},
\end{align*}
where the last inequality follows from $w \le 1$. 
Note that $|a - b| = \left|\int (w(x) - \alpha) (\mu(dx) - \nu(dx))\right| \le (1 - \alpha)|\mu - \nu|_{TV}$.
As a result,
\begin{equation*}
|\mu^w - \nu^w|_{TV} \le \left(\frac{1}{a} + \frac{1- \alpha}{2 a}\right)  |\mu - \nu|_{TV}
\le \left(\frac{1}{\alpha} + \frac{1- \alpha}{2 \alpha}\right)  |\mu - \nu|_{TV}.
\end{equation*}
\end{proof}

\quad Now let's go back to the proof of Theorem \ref{thm:finite}.
Applying Lemma \ref{lem:stab} with 
$w(x) =\prod_{i=1}^k \lambda_i(x_i) \ge \prod_{i=1}^k r_i$,
$\mu = P_{o,k}$ and $\nu = Q_{o,k}$ yields:
\begin{equation}
\label{eq:applem}
\left|Q_{\mathcal{U}_x, k}^{\widetilde{\mathcal{U}}_X^k} - Q_{\mathcal{U}_x, k}\right|_{TV}
\le \frac{3 - \prod_{i=1}^k r_i}{2 \prod_{i=1}^k r_i} |P_{o,k} - Q_{o,k}|_{TV}.
\end{equation}
Assume without loss of generality that $\mathcal{X} = \{1,\ldots, c\}$.
Let $n_j$ be the number of $j$'s in $\mathcal{U}_x$,
and $\nu_j$ be the number of $j$'s in the $k$-tuple $z$. 
It follows from \cite[Theorem 4]{DF80} that
\begin{equation}\label{eq:DF4}
|P_{o,k} - Q_{o,k}|_{TV} \le \frac{ck}{n}.
\end{equation}
Combining \eqref{eq:keytr}, \eqref{eq:applem} and \eqref{eq:DF4} yields \eqref{eq:boundfinite}.
\end{proof}

\quad Again if $P_n$ is exchangeable,
then the bound \eqref{eq:boundfinite} reduces to the Diaconis-Freedman bound (\cite[Theorem 3]{DF80}), 
which is sharp.

\quad As a consequence of Theorem \ref{thm:general}, 
we derive a first version of de Finetti's theorem for infinite weighted exchangeable sequences,
which is an analog to the Hewitt-Savage theorem \cite[Theorem 7.2]{HS55}.

\begin{corollary}
\label{thm:infinite}
Let $P$ be the distribution of a $\lambda$-exchangeable sequence $(X_1, X_2, \ldots)$ on $\mathcal{X}^{\infty}$.
Assume that
$\lambda_i$'s are bounded continuous functions such that
\begin{equation}
\label{eq:condratio}
r_i: = \frac{\inf_{\mathcal{X}} \lambda_i(x)}{\sup_{\mathcal{X}} \lambda_i(x)} > 0, \, \, \mbox{for } i =1,2,\ldots,
\quad \mbox{and} \quad \sum_{i = 1}^{\infty} (1- r_i) < \infty.
\end{equation}
Then $P$ is a mixture of weighted i.i.d. distributions on $\mathcal{X}^{\infty}$.
\end{corollary}
\begin{proof}
The proof is standard, and we provide it here for completeness. 
Without loss of generality,
assume  that $\sup_{\mathcal{X}} \lambda_i(x) = 1$ for each $i$,
so $r_i \le \lambda_i(x) \le 1$.
The assumption
\begin{equation*}
 \sum_{i = 1}^{\infty} (1- r_i) < \infty \quad \mbox{is equivalent to} \quad \prod_{i = 1}^\infty r_i > 0.
\end{equation*}

{\em Step 1}. Fix $k \le n$. 
Applying Theorem \ref{thm:general} to the $n$-marginal $P_n$ yields:
\begin{equation}
\label{eq:mn}
|P_k - Q_{n,k}|_{TV} \le \frac{k(k-1)}{2 n}\left(\prod_{i = 1}^k r_i \right)^{-1} + \frac{(\prod_{i = k+1}^n r_i)^{-\frac{1}{2}}- 1}{(\prod_{i = k+1}^n r_i)^{-\frac{1}{2}} +1},
\end{equation}
for $Q_{n,k}$ some mixture of weighted i.i.d. laws on $\mathcal{X}^k$.
%Taking $n \to \infty$, we get
%\begin{equation*}
%\limsup_{n \to \infty} |P_m - Q_{n,m}|_{TV} \le \frac{(\prod_{i = m+1}^\infty r_i)^{-\frac{1}{2}}- 1}{(\prod_{i = m+1}^\infty r_i)^{-\frac{1}{2}} %+1},
%\end{equation*}
%because $\prod_{i = 1}^m r_i \ge \prod_{i = 1}^\infty r_i > 0$.

{\em Step 2}. Choose $k_j \to \infty$ such that 
$\frac{(\prod_{i = k+1}^\infty r_i)^{-\frac{1}{2}}- 1}{(\prod_{i = k+1}^\infty r_i)^{-\frac{1}{2}} +1} \le \frac{1}{j}$.
For each $j$, choose $n_j > k_j$ sufficiently large that
\begin{equation*}
\frac{k_j(k_j-1)}{2 n_j }\left(\prod_{i = 1}^{k_j} r_i \right)^{-1} \le \frac{1}{j}.
\end{equation*}
By \eqref{eq:mn}, there exists $Q_j$ some mixture of weighted i.i.d. laws on $\mathcal{X}^{k_j}$ such that
\begin{equation}
\label{eq:2j}
|P_{k_j} - Q_j|_{TV} \le 2/j.
\end{equation}
Denote by $T_i \mu(dx) \propto \lambda_i(x) \mu(dx)$.
We can write
\begin{equation*}
Q_j = \int_{\mathcal{P}(\mathcal{X})} \otimes_{i = 1}^{k_j}T_i \mu \, \Pi_j(d\mu).
\end{equation*}
for some probability measure $\Pi_j$ on $\mathcal{P}(\mathcal{X})$.

{\em Step 3}. We claim that $\{\Pi_j\}$ is tight on $\mathcal{P}(\mathcal{X})$. 
First, let $Q_{j,1}$ be the first marginal of $Q_j$. 
We have $|P_1 - Q_{j,1}|_{TV} \le 2/j$,
and hence, $\{Q_{j,1}\}$ is tight. 
Since $r_1 \le \lambda_1(x) \le 1$, we get:
\begin{equation}
\label{eq:mur1}
\mu(A) \le \frac{1}{r_1} T_1 \mu(A) \quad \mbox{for any Borel set } A. 
\end{equation}
Fix $\varepsilon > 0$.
Because $\{Q_{j,1}\}$ is tight, we can choose increasing compact sets $K_{\ell} \subset \mathcal{X}$ such that
\begin{equation*}
\sup_j Q_{j,1}(K_{\ell}^c) \le r_1 \varepsilon 2^{-2 \ell}.
\end{equation*}
By \eqref{eq:mur1},
$\int  \mu(K_\ell^c) \Pi_j(d \mu) \le \frac{1}{r_1} Q_{j,1}(K_\ell^c) \le  \varepsilon 2^{-2 \ell}$.
Markov's inequality yields $\Pi_j\left\{\mu: \mu(K_\ell^c) > 2^{-\ell} \right\} \le \varepsilon 2^{-\ell}$.
Set 
\begin{equation*}
\mathcal{K}:= \cap_{\ell \ge 1}\left\{\mu: \mu(K_\ell^c) \le 2^{-\ell} \right\}.
\end{equation*}
We have $\Pi_j(\mathcal{K}) \ge 1- \varepsilon$,
so $\mathcal{K}$ is uniformly tight in $\mathcal{P}(\mathcal{X})$.
This implies the tightness of $\{\Pi_j\}$.
Passing to a subsequence, 
$\Pi_j$ converges weakly to some probability measure $\Pi$ on $\mathcal{P}(\mathcal{X})$.

{\em Step 4}. Since $\lambda_i$'s are bounded continuous function, 
$\mu \to T_i \mu$ is weakly continuous. 
So for each fixed $m$, 
\begin{equation*}
\mu \to \otimes_{i=1}^m T_i \mu \mbox{ is weakly continuous}.
\end{equation*}
Fix $m$.
We have $k_j \ge m$ for sufficiently large $j$.
Let $Q_{j,k}$ be the first $k$ marginal of $Q_j$,
so
\begin{equation*}
Q_{j,m} = \int_{\mathcal{P}(\mathcal{X})} \otimes_{i=1}^m T_i \mu \, \Pi_j(d\mu).
\end{equation*}
By \eqref{eq:2j}, we get $|P_m - Q_{j,m}|_{TV} \le 2/j$.
Thus, $Q_{j,m}$ converges to $P_m$ as $j \to \infty$.
On the other hand,
the convergence of $\Pi_j$ to $\Pi$ and the weak continuity of $\mu \to \otimes_{i=1}^m T_i \mu$ yields
\begin{equation*}
Q_{j,m} \mbox{ converges to } \int_{\mathcal{P}(\mathcal{X})} \otimes_{i=1}^m T_i \mu \, \Pi(d\mu).
\end{equation*}
Therefore, $P_m = \int_{\mathcal{P}(\mathcal{X})} \otimes_{i=1}^m T_i \mu \, \Pi(d\mu)$ for each $m$,
and $P = \int_{\mathcal{P}(\mathcal{X})} \otimes_{i=1}^\infty T_i \mu \, \Pi(d\mu)$.
%The proof follows from \cite[Theorem 14]{DF80} and Theorem \ref{thm:general}.
%First assume that $\mathcal{X}$ is compact.
%By Theorem \ref{thm:general}, for each $n$,
%there exists a mixture of weighted i.i.d distribution $Q_n$ on $\mathcal{X}^n$ such that
%\begin{equation}
%\label{eq:reest}
%|P_k - Q_{n,k}|_{TV} \le \frac{k(k-1)}{2n} \left( \prod_{i = 1}^k r_i\right)^{-1},
%\end{equation}
%where $Q_{n,k}$ denotes the $k$-marginal distribution of $Q_n$.
%Denote by $(\mathcal{X}^*, \mathcal{B}^*)$ the set of distributions on $(\mathcal{X}, \mathcal{B})$,
%and $(\mathcal{X}^{**}, \mathcal{B}^{**})$ the set of distributions on $(\mathcal{X}^*, \mathcal{B}^*)$.
%By the compactness of $(\mathcal{X}^{**}, \mathcal{B}^{**})$,
%$Q_{n,k}$ converges in the weak-star topology to some $Q^*_k$,
%which is the $k$-marginal of a mixture of weighted i.i.d. distributions $Q^*$ on $\mathcal{X}^{\infty}$. 
%\quad It is known that the condition \eqref{eq:condratio} is equivalent to 
%$\prod_{i = 1}^\infty r_i > 0$. 
%Combining with \eqref{eq:reest}, 
%we get the convergence of $Q_{n,k}$ to $P_k$ in total variation. 
%As a result, $P_k = Q^*_k$ for all $k$, and $P = Q^*$.
%By a routine argument in \cite{Ressel85}, 
%the assumption that $\mathcal{X}$ is compact
%can be extended to Borel subsets of compact spaces, 
%including Polish spaces. 
%This completes the proof.
\end{proof}

\quad As mentioned in the introduction (see \eqref{eq:gencond0}),
\cite[Theorem 6]{BC23} showed that if
\begin{equation}
\label{eq:BCcond}
\sum_{i =1}^{\infty} r_i = \infty,
\end{equation}
then $P$ is a mixture of weighted i.i.d. distributions. 
Note that the condition \eqref{eq:condratio} is more restrictive than \eqref{eq:BCcond}.
Essentially, it implies that an infinite ``nearly'' exchangeable sequence
is a mixture of weighted i.i.d. sequences. 
This calls for an alternative factor to $\prod_{i = 1}^k r_i$ in \eqref{eq:boundgen}, 
so that de Finetti's representation for infinite weighted i.i.d. sequences
can be derived from finite approximations under broader sufficient (and possibly necessary) conditions 
such as \eqref{eq:gencond0}.
This is the goal of the next section.
Also
the assumption that $\mathcal{X}$ is a Polish space may be relaxed. 
But without any topological assumptions on $\mathcal{X}$,
de Finetti's theorem for exchangeable sequences even fails \cite{DF79, Free80}. 

\section{A probabilistic approach}
\label{sc4}

\quad In the section, we present a probabilistic approach to the finite weighted de Finetti theory,
which yields an alternative proof of \cite[Theorem 6]{BC23}.
We rely on a thinning argument, together with stochastic comparisons.

\begin{theorem}
\label{thm:general2}
Let $P_n$ be the distribution of a $\lambda$-exchangeable sequence $(X_1, \ldots, X_n)$ on $\mathcal{X}$,
and $P_k$ be the marginal distribution of $(X_1, \ldots, X_k)$ for $k \le n$.
Assume that \eqref{eq:assumpr} holds.
Then there exists a mixture of weighted i.i.d. distributions $Q_n$ on $\mathcal{X}^n$, 
with $Q_{n,k}$ the marginal distribution on its first $k$ coordinates such that
\begin{equation}
\label{eq:boundgen2}
|P_k - Q_{n,k}|_{TV} \le \frac{k(k-1)}{2 \min_{i \le k} r_i}\frac{1-e^{-\sum_{i = k+1}^n r_i}}{\sum_{i = k+1}^n r_i}.
\end{equation}
\end{theorem}
\begin{proof}
Without loss of generality, 
assume that $\sup_{x\in\mathcal X}\lambda_i(x)=1$ for each $i$,
so $r_i\le \lambda_i(x)\le 1$.

\quad Fix $k < n$, and introduce i.i.d. $U_{k+1},\ldots,U_n \sim \mathrm{Unif}(0,1)$ independent of $(X_1, \ldots, X_n)$.
Define
\begin{equation*}
B_i: = 1\left( U_i\le \frac{r_i}{\lambda_i(X_i)} \right) \quad \mbox{for } i = k+1,\ldots,n 
\quad \mbox{and} \quad
A: = \{i \in [k+1, n]: B_i = 1\}.
\end{equation*}
We claim that conditional on $A$, the ``retained" coordinates $J: = [k] \cup A$
is weighted exchangeable with weights
\begin{equation*}
\lambda_1,\ldots,\lambda_k, \underbrace{1,\ldots,1}_{\mbox{\tiny card}(A)\text{ times}}.
\end{equation*}
Recall that 
$P_n(dx_1\cdots dx_n) = \left(\prod_{i=1}^n \lambda_i(x_i)\right) \overline{P}_n(dx_1\cdots dx_n)$,
where $\overline{P}_n$ is some exchangeable probability measure.
Fixing $A \subset [k+1, n]$,
the joint distribution of $X$ and the event $\{B_i = 1(i \in A)\}$
is proportional to 
\begin{equation*}
\left(\prod_{i=1}^k\lambda_i(x_i)\right) \left(\prod_{i\in A}r_i\right) \left(\prod_{i\in J^c}(\lambda_i(x_i)-r_i)\right) \overline{P}_n(dx_1\cdots dx_n). 
\end{equation*}
Integrating out the ``rejected" coordinates $J^c$, 
the retained coordinates have the distribution of form:
\begin{equation*}
\left(\prod_{i=1}^k\lambda_i(x_i)\right) P_A(dx_J), 
\quad \mbox{where }
P_A(dx_J) := \int_{\mathcal X^{J^c}} \left(\prod_{i\in J^c}(\lambda_i(x_i)-r_i)\right) \overline{P}_n(dx_J,dx_{J^c}).
\end{equation*}
Clearly, $P_A$ is symmetric in the coordinates indexed by $J$,
and the claim follows.

\quad Next condition on a (deterministic) retained urn,
and by Proposition \ref{prop:finiteexrep},
it suffices to consider the corresponding extremal weighted urn distribution $P_{\mathcal{U}_y}$ on
\begin{equation*}
\mathcal{U}_y: = \{y_1, \ldots, y_{k + \mbox{\tiny card}(A)}\}.
\end{equation*}
Because all weights after the first coordinates are equal to $1$, 
every completion of a fixed distinct $k$-tuple has the same weight $(N-k)!$.
Therefore, 
\begin{equation}\label{eq:PUy}
P_{\mathcal{U}_y, k}(j_1, \ldots, j_k) \propto \prod_{i=1}^k \lambda_i(y_{j_i}), \quad 
j_1, \ldots, j_k \mbox{ distinct}.
\end{equation}
Define the weighted sampling-with-replacement distribution $Q_{\mathcal{U}_y, k}$
by choosing 
labels $I_1, \ldots, I_k \in [k + \mbox{card}(A)]$ independently according to
\begin{equation*}
\mathbb P(I_i=j) = \frac{\lambda_i(y_j)} {\sum_{\ell=1}^{k + \mbox{\tiny card}(A)}\lambda_i(y_\ell)}, \qquad j=1,\ldots, k+\mbox{card}(A).
\end{equation*}
As a result, 
\begin{equation}\label{eq:QUy}
Q_{\mathcal{U}_y, k}(j_1, \ldots, j_k) = \prod_{i=1}^k \frac{\lambda_i(y_{j_i})}{\sum_{\ell=1}^{k + \mbox{\tiny card}(A)}\lambda_i(y_\ell)}.
\end{equation}
From \eqref{eq:PUy} and \eqref{eq:QUy},
we see that $P_{\mathcal{U}_y, k}$ has the same distribution as $Q_{\mathcal{U}_y, k}$ conditioned on
$\{I_1, \ldots, I_k \mbox{ distinct}\}$.
Therefore,
\begin{equation}\label{eq:union}
\begin{aligned}
|P_{\mathcal{U}_y, k} - Q_{\mathcal{U}_y, k}|_{TV}
& \le Q_{\mathcal{U}_y, k}(\{I_1,\ldots,I_k \text{ distinct}\}^c) \\
& \le \sum_{1\le i<j\le k} \mathbb P(I_i=I_j),
\end{aligned}
\end{equation}
where the last inequality is obtained by union bound.
Note that for $i < j$,
\begin{equation}\label{eq:IJ}
\begin{aligned}
\mathbb P(I_i=I_j) & = \frac{ \sum_{\ell=1}^{k + \mbox{\tiny card}(A)} \lambda_i(y_\ell)\lambda_j(y_\ell) }{\left(\sum_{\ell=1}^{k + \mbox{\tiny card}(A)}\lambda_i(y_\ell)\right) \left(\sum_{\ell=1}^{k + \mbox{\tiny card}(A)}\lambda_j(y_\ell)\right)} \\
& \le \frac{ \min\left\{\sum_{\ell=1}^{k + \mbox{\tiny card}(A)}\lambda_i(y_\ell), \sum_{\ell=1}^{k + \mbox{\tiny card}(A)}\lambda_j(y_\ell)\right\} }{\left(\sum_{\ell=1}^{k + \mbox{\tiny card}(A)}\lambda_i(y_\ell)\right) \left(\sum_{\ell=1}^{k + \mbox{\tiny card}(A)}\lambda_j(y_\ell)\right)} \\
& \le \frac{1}{(k + \mbox{card}(A)) \max(r_i, r_j)}.
\end{aligned}
\end{equation}
It follows from \eqref{eq:union} and \eqref{eq:IJ} that
\begin{equation*}
|P_{\mathcal{U}_y, k} - Q_{\mathcal{U}_y, k}|_{TV} \le \frac{1}{k + \mbox{card}(A)} C_k(r),
\quad \mbox{where } C_k(r): = \sum_{1 \le i < j \le k} \max(r_i, r_j)^{-1}.
\end{equation*}
Averaging first over the retained-urn mixture and then over the thinning pattern yields a mixture $Q_{n,k}$ of weighted i.i.d. laws such that 
\begin{equation}\label{eq:TVprob}
|P_k - Q_{n,k}|_{TV} \le C_k(r) \mathbb{E}\left(\frac{1}{k + \mbox{card}(A)} \right).
\end{equation}
Recall that $\frac{r_i}{\lambda_i(X_i)} \le r_i$.
Define $Z_i$ to be independent Bernoulli$(r_i)$, so
$\mbox{card}(A) \ge \sum_{i=k+1}^n Z_i$ almost surely by an obvious coupling. 
The sum $\sum_{i=k+1}^n Z_i$ is Poisson binomial.
Consequently, 
\begin{equation}\label{eq:domin}
\begin{aligned}
\mathbb{E}\left(\frac{1}{k + \mbox{card}(A)} \right) & 
\le \mathbb{E} \left(\frac{1}{k +  \sum_{i=k+1}^n Z_i }\right) \\
& = \int_0^1 t^{k-1} \prod_{i=k+1}^n (1-r_i+r_it)\,dt \\ 
&\le \int_0^1 t^{k-1} \exp\left(-(1-t)\sum_{i = k+1}^n r_i\right)\,dt \\ &\le \int_0^1 \exp\left(-(1-t)\sum_{i = k+1}^n r_i\right)\,dt \\ 
&= \frac{1-e^{-\sum_{i = k+1}^n r_i}}{\sum_{i = k+1}^n r_i},
\end{aligned}
\end{equation}
where the second equality follows from the identity $\frac{1}{k+z} = \int_0^1 t^{k+z-1} dt$
and the probability generating function of Poisson binomial \cite[(16)]{TT23}.
Combining \eqref{eq:TVprob} and \eqref{eq:domin} yields \eqref{eq:boundgen2}.
\end{proof}

\quad The key to Theorem \ref{thm:general2} is the probabilistic bound \eqref{eq:TVprob}.
If $P_n$ is exchangeable,
then this bound reduces to the Diaconis-Freedman bound $\frac{k(k-1)}{2n}$.
This is slightly stronger than the bound $\frac{k(k-1)}{2(n-k)}$ in \eqref{eq:boundgen2}, 
and the loss comes from replacing the deterministic retained count $n-k$ by the looser Poisson binomial estimate.

\quad Finally, following the same argument in Corollary \ref{thm:infinite} (from Theorem \ref{thm:general}), 
we get the weighted de Finetti theorem 
under the weaker assumption $\sum_{i = 1}^\infty r_i = \infty$,
which recovers \cite[Theorem 6]{BC23}.
\begin{corollary}
\label{thm:infinite2}
Let $P$ be the distribution of a $\lambda$-exchangeable sequence $(X_1, X_2, \ldots)$ on $\mathcal{X}^{\infty}$.
Assume that
$\lambda_i$'s are bounded continuous functions such that
\begin{equation}
\label{eq:condratio}
r_i: = \frac{\inf_{\mathcal{X}} \lambda_i(x)}{\sup_{\mathcal{X}} \lambda_i(x)} > 0, \, \, \mbox{for } i =1,2,\ldots,
\quad \mbox{and} \quad \sum_{i = 1}^{\infty}  r_i = \infty.
\end{equation}
Then $P$ is a mixture of weighted i.i.d. distributions on $\mathcal{X}^{\infty}$.
\end{corollary}

\bigskip
{\bf Acknowledgement:} 
W.\ Tang is supported by NSF CAREER Award DMS-2538791 and the Tang Family Assistant Professorship.

\bibliographystyle{abbrv}
\bibliography{unique}
\end{document}